\documentclass[a4paper]{article}
\usepackage{amsmath}

\newtheorem{thm}{Theorem}

\newtheorem{lem}{Lemma}
\newtheorem{prob}{Problem}

\newtheorem{rmk}{Remark}
\newtheorem{example}{Example}
\newenvironment{proof}[1][Proof]{\noindent\textbf{#1.} }{\ \rule{0.5em}{0.5em}}
\numberwithin{equation}{section}

\begin{document}

\title{\textbf{On The Best Approximate Solutions of The Matrix Equation} $%
AXB=C$}
\author{Halim \"{O}zdemir, Murat Sarduvan \\
%EndAName
{\small Department of Mathematics, Sakarya University, TR54187 Sakarya,
Turkey}}
\date{}
\maketitle

\begin{abstract}
Suppose that the matrix equation $AXB=C$ with unknown matrix $X$ is given,
where $A$, $B$, and $C$\ are known matrices of suitable sizes. The matrix
nearness problem is considered over the general and least squares solutions
of the matrix equation $AXB=C$ when the equation is consistent and inconsistent,
respectively. The implicit form of the best approximate solutions of the
problems over the set of symmetric and the set of skew-symmetric matrices are
established as well. Moreover, some numerical examples are given for the
problems considered.
\end{abstract}

%%% ----------------------------------------------------------------------



\section{Introduction and Notations}

Let $\mathcal{R}^{m\times n}$, $\mathcal{SR}^{n\times n}$, and $\mathcal{SSR}%
^{n\times n}$ be the set of $m\times n$ real matrices, the set of $n\times n$
\ real symmetric matrices, and the set of $n\times n$\ real skew-symmetric
matrices, respectively. The symbols $A^{T}$, $A^{{\dag }}$, and $\left\Vert
A\right\Vert $\ will denote the transpose, the Moore-Penrose generalized
inverse, and the Frobenius norm (see, for example, \cite{Horn}),
respectively, of a matrix $A\in {\mathcal{\ R}^{m\times n}}$. Further, $%
vec\left( \cdot \right) $ will stand for the $vec$ operator, i.e. $vec\left(
A\right) ={\left( {a_{1}^{T},a_{2}^{T},\ldots ,a_{n}^{T}}\right) ^{T}}$ for
the matrix $A=\left( {{a_{1}},{a_{2}},\ldots ,{a_{n}}}\right) \in {\mathcal{R%
}^{m\times n}}$, ${a_{i}}\in {R^{m\times 1}}$, $i=1,2,\ldots ,n,$ and $%
A\otimes B$ will stand for the Kronecker product of matrices $A$ and $B\in {%
\mathcal{R}^{m\times n}}$, (see \cite{Graybill}).

The well-known linear matrix equation $AXB=C$, where $A$, $B$, $C$ are known
matrices of suitable sizes and $X$ is the matrix of unknowns, were studied
in the case of special solution structures, e.g. symmetric, triangular or
diagonal solution $X$ in \cite{Bjerhammer, Chu, Dai, Don, Magnus, Morris} using matrix
decomposition such as the singular value decomposition (SVD), the
generalized SVD, the quotient SVD, and the canonical correlation
decomposition. In these literatures, the matrix equation $AXB=C$ is
consistent. But, it is rarely possible to satisfy the consistency condition
of the matrix equation $AXB=C$, since the matrices $A$, $B$, and $C$
occurring in practice are usually obtained from an experiment.

An iteration method to solve the linear matrix equation $AXB=C$ over the set
of symmetric matrices have constructed by Peng et al. \cite{YXPeng}. In
addition, Peng \cite{ZYPeng2} has established an iterative method to solve
the minimum Frobenius norm residual problem: $\min \left\Vert {AXB-C}%
\right\Vert $ where $X$ is the symmetric matrix of unknowns. Huang and Yin
solved the constrained inverse eigenproblem and associated approximation
problem for anti-Hermitian $R$-symmetric matrices and the matrix inverse
problem and its optimal approximation problem for $R$-symmetric matrices in
\cite{Huang1} and \cite{Huang2}, respectively. Huang et al. gave the precise
solutions to the minimum residual problem and the matrix nearness problem
for symmetric matrices or skew-symmetric matrices in \cite{Huang3} and
constructed an iterative method to solve the linear matrix equation $AXB=C$
over the set of skew-symmetric matrices in \cite{Huang4}.

This work is devoted to give the best approximate solutions of the following %
two problems, which are interesting and known as the matrix nearness problems, %
in an alternative way:

\begin{prob}
\label{prob1} For given matrices $A \in {\mathcal{R}^{m \times n}}$, $B \in {%
\ \mathcal{R}^{p \times r}}$, and $C \in {\mathcal{R}^{m \times r}}$, let ${%
S_G }$ be the set of all solutions of the consistent matrix equation
\[
AXB = C.
\]
For a given matrix $X_0 \in {\mathcal{R}^{n \times p}}$, find $\hat X \in {\
S_G}$ such that
\[
\left\| {\hat X - {X_0}} \right\| = \mathop {\min }\limits_{X \in {S_G}}
\left\| {X - {X_0}} \right\| .
\]
\end{prob}

\begin{prob}
\label{prob2} For given matrices $A\in {\mathcal{R}^{m\times n}}$, $B\in {\
\mathcal{R}^{p\times r}}$, and $C\in {\mathcal{R}^{m\times r}}$, let ${S_{E}}
$ be the set of all least squares solutions of the minimum residual problem%
\[
\mathop {\min }\limits_{X\in {\mathcal{R}^{n\times p}}}\left\Vert {AXB-C}%
\right\Vert .
\]
For a given matrix $X_{0}\in {\mathcal{R}^{n\times p}}$, find $\hat{X}\in {%
S_{E}}$ such that
\[
\left\Vert {\hat{X}-{X_{0}}}\right\Vert =\mathop {\min }\limits_{X\in {S_{E}}%
}\left\Vert {X-{X_{0}}}\right\Vert .
\]
\end{prob}

In fact, the Problems \ref{prob1} and \ref{prob2} are to find the best
approximate solution for a given matrix $X_{0}\in {\mathcal{R}^{n\times p}}$
over the set of general solutions (${S_{G}}$) and the least square solutions
(${S_{E}}$) of the matrix equation $AXB=C$, respectively. These problems are
known as the matrix nearness problem in the literature. The matrix nearness
problem is very important in applied sciences and has been extensively
studied in recent years (see, for example, \cite{Higham, Jeseph, Jiang, ZYPeng1}).
Therefore, it is important to give the best approximate solutions of the
problems in implicit forms.

In general, numerical algorithms or iteration methods for solving these
problems are suggested in most of the works mentioned above.

In this work, the implicit forms of the best approximate solutions to the
problems mentioned above have been obtained over the set of symmetric and
the set of skew--symmetric matrices using the Moore--Penrose generalized
inverse. Moreover, some numerical examples are given via Matlab 7.5. The
matrices in the examples have been taken from related reference work articles.

\section{Preliminary Results}

The vector $x_{0}\in {\mathcal{R}^{n\times 1}}$ is the best approximate
solution (BAS) to the inconsistent system of linear equations $Ax=g$, where $%
A\in {\mathcal{R}^{m\times n}}$, if and only if

\begin{description}
\item[1-] ${\left( {Ax-g}\right) ^{T}}\left( {Ax-g}\right) \geq {\left( {A{%
x_{0}}-g}\right) ^{T}}\left( {A{x_{0}}-g}\right) $ for all $x\in {\mathcal{R}%
^{n\times 1}}$,

\item[2-] ${x^{T}}x>{x_{0}}^{T}{x_{0}}$ for all $x\in {\mathcal{R}^{n\times
1}\backslash }\left\{ {x_{0}}\right\} $ satisfying \ \ ${\left( {Ax-g}%
\right) ^{T}}\left( {Ax-g}\right) ={\left( {A{x_{0}}-g}\right) ^{T}}\left( {A%
{x_{0}} -g}\right) $ \cite{Graybill}.
\end{description}

The vector $x_{0}\in {\mathcal{R}^{n\times 1}}$ is a least squares solution
(LSS) to the inconsistent system of linear equations $Ax=g$, where $A\in {%
\mathcal{R}^{n\times 1}}$, if and only if

\[
{\left( {Ax - g} \right)^T}\left( {Ax - g} \right) \ge {\left( {A{x_0} - g}
\right)^T}\left( {A{x_0} - g} \right)
\]
for all $x\in {\mathcal{R}^{n\times 1}}$ \cite{Graybill}.

It is noteworthy that there may be many LSS for an inconsistent system of
linear equations. In addition, an LSS may not be the BAS while the BAS is
always a LSS. However, the BAS is always unique.

We close this section by giving two auxiliary results related to the
problems mentioned earlier and which will be used in the rest of the work.

\begin{lem}
\label{lem:one} Suppose that ${S_{G^{\prime }}}$ is the set of all solutions
to the consistent system of linear equations $Ax=g$, where $A\in {\mathcal{R}%
^{m\times n}}$ is a known matrix, $g\in {\mathcal{R}^{m\times 1}}$ is a
known vector, and $x\in {\mathcal{R}^{n\times 1}}$ is the vector of
unknowns. For a given vector $x_{0}\in {\mathcal{R}^{n\times 1}}$, the
vector $\hat{x}\in {S_{G}^{\prime }}$ satisfying
\[
\left\Vert {\hat{x}-{x_{0}}}\right\Vert =\mathop {\min }\limits_{x\in {\
S_{G^{\prime }}}}\left\Vert {x-{x_{0}}}\right\Vert
\]%
is given by
\[
\hat{x}={A^{\dag }}g+\left( {I-{A^{\dag }}A}\right) {x_{0}}.
\]
\end{lem}

\begin{proof}
If $x\in {S_{G^{\prime }}}$, then it can be written in the form
\[
x={A^{\dag }}g+\left( {I-{A^{\dag }}A}\right) h
\]%
for some vector $h\in {\mathcal{R}^{n\times 1}}$ \cite[Theorem 6.3.2]{Graybill}.
Thus, the problem turns into the problem of finding the BAS $\hat{x}$ of the
system
\begin{equation}
{A^{\dag }}g+\left( {I-{A^{\dag }}A}\right) h={x_{0}}  \label{eq:x0}
\end{equation}%
or equivalently the system
\[
\left( {I-{A^{\dag }}A}\right) h={x_{0}}-{A^{\dag }}g.
\]%
Since the matrix ${I-{A^{\dag }}A}$ is symmetric and idempotent, it is
obtained
\[
\hat{h}=\left( {I-{A^{\dag }}A}\right) {x_{0}}
\]%
by Theorem 7.4.1 in \cite{Graybill}. Substituting this expression in the
equation (\ref{eq:x0}), we get
\[
\hat{x}={A^{\dag }}g+\left( {I-{A^{\dag }}A}\right) {x_{0}}.
\]%
So, the proof is completed.
\end{proof}

\begin{lem}
\label{lem:two} Let ${S_{E^{\prime }}}$ be the set of all least squares
solutions to the system of linear equations $Ax=g$ which do not need to be
consistent, where $A\in {\mathcal{R}^{m\times n}}$ is a known matrix, $g\in {%
\mathcal{R}^{m\times 1}}$ is a known vector, and $x\in {\mathcal{R}^{n\times
1}}$ is the vectors of unknowns. For a given vector $x_{0}\in {\mathcal{R}%
^{n\times 1}}$, the vector $\hat{x}\in {\ S_{E}^{\prime }}$ satisfying
\[
\left\Vert {\hat{x}-{x_{0}}}\right\Vert =\mathop {\min }\limits_{x\in {\
S_{E^{\prime }}}}\left\Vert {x-{x_{0}}}\right\Vert
\]%
is given by
\[
\hat{x}={A^{\dag }}g+\left( {I-{A^{\dag }}A}\right) {x_{0}}.
\]
\end{lem}

\begin{proof}
If $x\in {S_{E^{\prime }}}$, then it can be written in the form
\[
x={A^{\dag }}g+\left( {I-{A^{\dag }}A}\right) h
\]%
for some vector $h\in {\mathcal{R}^{n\times 1}}$ \cite[Theorem 6.3.2]{Graybill}.
This is of the same type with (\ref{eq:x0}). So, the remaining part of the
proof can be completed easily in a similar way as in the proof of Lemma \ref%
{lem:one}.
\end{proof}\\

It is noteworthy that the structures of $\hat x$ in Lemmas \ref{lem:one} and %
\ref{lem:two} are exactly the same.

\section{The Best Approximate Solutions of Problems 1 and 2}

If it is assumed that the matrix equation $AXB=C$, where $A\in {\mathcal{R}%
^{m\times n}}$, $B\in {\mathcal{R}^{p\times r}}$, $C\in {\mathcal{R}%
^{m\times r}}$ are known nonzero matrices and $X\in {\mathcal{R}^{n\times p}}
$ is the matrix of unknowns, is inconsistent, as was the system of linear
equations, then it may be asked to find a matrix $X$ such that $\left\Vert {%
AXB-C}\right\Vert $ is minimum, too. A matrix satisfying this condition is
called an approximate solution to the matrix equation. The matrix $\hat{X}%
\in {\mathcal{R}^{n\times p}}$ is defined to be the BAS to the matrix
equation $AXB=C$ if and only if

\begin{description}
\item[1-] $\left\Vert {AXB-C}\right\Vert \geq \left\Vert {A{\hat{X}}B-C}%
\right\Vert$ for all $X\in {\mathcal{R}^{n\times p}}$;

\item[2-] $\left\Vert {X}\right\Vert >\left\Vert {\hat{X}}\right\Vert $ for
all $X\in {\mathcal{R}^{n\times p}\backslash }\left\{ {\hat{X}}\right\} $
satisfying $\left\Vert {AXB-C}\right\Vert =\left\Vert {A{\hat{X}}B-C}%
\right\Vert $.
\end{description}

We note that a vector $k\in {\mathcal{R}^{mn\times 1}}$ will stand for the
vector $vec(K)$ in the rest of the text, where $K\in {\mathcal{R}^{m\times n}%
}$.

It is known that the matrix equation $AXB=C$ can be equivalently written as
\begin{equation}
\left( {{B^{T}}\otimes A}\right) x=c,  \label{eq:kron}
\end{equation}%
where $\left( {{B^{T}}\otimes A}\right) $ is the kronecker product (see, for
detail, \cite{Rao}). Consequently, the solutions of a matrix equation $AXB=C$
can be obtained considering the usual system of linear equations (\ref%
{eq:kron}) instead of the matrix equation $AXB=C$.

Now we can give the solutions of Problems \ref{prob1} and \ref{prob2} which
are the subjects of the following two theorems, respectively.

\begin{thm}
\label{thm:one} Let the matrix equation $AXB=C$ be consistent, where $A\in {%
\mathcal{R}^{m\times n}}$, $B\in {\mathcal{R}^{p\times r}}$, $C\in {\mathcal{%
R}^{m\times r}}$ are known nonzero matrices and and $X\in {\mathcal{R}%
^{n\times p}}$ is the matrix of unknowns. Then, for a given matrix $X_{0}\in
{\mathcal{R}^{n\times p}}$, the matrix $\hat{X}\in {S_{G}}$ satisfying
\[
\left\Vert {\hat{X}-{X_{0}}}\right\Vert =\mathop {min}\limits_{X\in {S_{G}}%
}\left\Vert {X-{X_{0}}}\right\Vert
\]
is given by
\[
\hat{X}={A^{\dag }}C{B^{\dag }}+{X_{0}}-{A^{\dag }}A{X_{0}}B{B^{\dag }}.
\]
\end{thm}

\begin{proof}
If $X\in {S_{G}}$, then it can be written in the form
\begin{equation}
X={A^{\dag }}C{B^{\dag }}+H-{A^{\dag }}AHB{B^{\dag }}  \label{eq:10}
\end{equation}%
for some matrix $H\in {\mathcal{R}^{n\times p}}$ \cite{Rao}. Since the
statement (\ref{eq:10}) is equivalent to the statement
\[
x={\left( {{B^{T}}\otimes A}\right) ^{\dag }}c+\left[ {I-{{\left( {{B^{T}}%
\otimes A}\right) }^{\dag }}\left( {{B^{T}}\otimes A}\right) }\right] h,
\]%
the problem turns into the problem of finding the best approximate solution
of the usual system of linear equations
\[
{\left( {{B^{T}}\otimes A}\right) ^{\dag }}c+\left[ {I-{{\left( {{B^{T}}%
\otimes A}\right) }^{\dag }}\left( {{B^{T}}\otimes A}\right) }\right] h={%
x_{0}}
\]%
or equivalently
\begin{equation}
\left[ {I-{{\left( {{B^{T}}\otimes A}\right) }^{\dag }}\left( {{B^{T}}%
\otimes A}\right) }\right] h={x_{0}}-{\left( {{B^{T}}\otimes A}\right)
^{\dag }}c.  \label{eq:20}
\end{equation}%
Using Lemma \ref{lem:one} and (\ref{eq:20}), we get
\[
\hat{x}={\left( {{B^{T}}\otimes A}\right) ^{\dag }}c+\left[ {I-{{\left( {{\
B^{T}}\otimes A}\right) }^{\dag }}\left( {{B^{T}}\otimes A}\right) }\right] {%
x_{0}}
\]%
or, in the matrix form,
\[
\hat{X}={A^{\dag }}C{B^{\dag }}+{X_{0}}-{A^{\dag }}A{X_{0}}B{B^{\dag }}.
\]%
So, the proof is completed.
\end{proof}

\begin{lem}
\label{lem:three} Let the matrices $A$, $B$, and $C$ be as in Theorem \ref%
{thm:one}, and assume that the matrix $X$ is a least squares solution to the
inconsistent matrix equation $AXB=C$. Then the matrix $X$ can be written in
the form
\[
X={A^{\dag }}C{B^{\dag }}+H-{A^{\dag }}AHB{B^{\dag }}
\]%
for some matrix $H\in {\mathcal{R}^{n\times p}}$.
\end{lem}

\begin{proof}
The proof is immediately follows from Theorem 6.3.2 and Theorem 7.6.3 in
\cite{Graybill} considering the usual system of linear equations ${{\left( {{%
B^{T}}\otimes A}\right) }}x=c$ instead of the matrix equations $AXB=C$
because the former and the latter are equivalent.
\end{proof}

\begin{thm}
\label{thm:two} Let the matrix equation $AXB=C$ be inconsistent, where $A\in
{\mathcal{R}^{m\times n}}$ , $B\in {\mathcal{R}^{p\times r}}$, $C\in {%
\mathcal{R}^{m\times r}}$ are known nonzero matrices and $X\in {\mathcal{R}%
^{n\times p}}$ is the matrix of unknowns. Then, for a given matrix $X_{0}\in
{\mathcal{\ R}^{n\times p}}$, the matrix $\hat{X}\in {S_{E}} $ satisfying
\[
\left\Vert {\hat{X}-{X_{0}}}\right\Vert =\mathop {min}\limits_{X\in {S_{E}}%
}\left\Vert {X-{X_{0}}}\right\Vert
\]%
is given by
\[
\hat{X}={A^{\dag }}C{B^{\dag }}+{X_{0}}-{A^{\dag }}A{X_{0}}B{B^{\dag }}.
\]
\end{thm}

\begin{proof}
By Lemma \ref{lem:three}, any least squares solutions of the inconsistent
matrix equation $AXB=C$ is in the form
\[
X={A^{\dag }}C{B^{\dag }}+H-{A^{\dag }}AHB{B^{\dag }}
\]%
for some matrix $H\in {\mathcal{R}^{n\times p}}$. Hence, in the framework of
Lemma \ref{lem:two}, the proof is easily completed by proceeding as in the
proof of Theorem \ref{thm:one}.
\end{proof}

\section{The Symmetric and Skew--Symmetric Solutions of Problems \protect\ref%
{prob1} and \protect\ref{prob2}}

Now, suppose that the symmetric solutions of Problems \ref{prob1} and \ref%
{prob2} are required. To do this, the pair of matrix equations
\[
\begin{array}{l}
AXB=C \\
{B^{T}}X{A^{T}}={C^{T}}%
\end{array}%
\]%
or, equivalently, the usual system of linear equations
\[
\left[ {\
\begin{array}{c}
{{B^{T}}\otimes A} \\
{A\otimes {B^{T}}}%
\end{array}%
}\right] x=\left[ {\
\begin{array}{c}
{{c_{1}}} \\
{{c_{2}}}%
\end{array}%
}\right]
\]%
instead of the matrix equation $AXB=C$ is taken with ${c_{1}}=vec\left(
C\right) $ and ${c_{2}}=vec\left( C^{T}\right) $. Then, for a given matrix ${%
\ \ X_{0}}\in S{\mathcal{R}^{n\times n}}$, in the framework of Theorems \ref%
{thm:one} and \ref{thm:two} the solution matrix $\hat{X} \in S{\mathcal{R}%
^{n\times n}}$ is obtained by using $\hat{x}=vec\left( {\hat{X}}\right) $
given by
\begin{equation}
\hat{x}={\left[ {\
\begin{array}{c}
{{B^{T}}\otimes A} \\
{A\otimes {B^{T}}}%
\end{array}%
}\right] ^{\dag }}\left[ {\
\begin{array}{c}
{{c_{1}}} \\
{{c_{2}}}%
\end{array}%
}\right] +{x_{0}}-{\left[ {\
\begin{array}{c}
{{B^{T}}\otimes A} \\
{A\otimes {B^{T}}}%
\end{array}%
}\right] ^{\dag }}\left[ {\
\begin{array}{c}
{{B^{T}}\otimes A} \\
{A\otimes {B^{T}}}%
\end{array}%
}\right] {x_{0}}.  \label{eq:coz1}
\end{equation}

\begin{rmk}
\label{rem:one} If the matrix $X_{0}$ is not symmetric, then the matrix ${%
\frac{1}{2}\left( {{X_{0}}+X_{0}^{T}}\right) }$, instead of the matrix $%
X_{0} $, is taken to find the symmetric solutions of Problems \ref{prob1} or %
\ref{prob2}. The reason for this is that the minimization problem%
\[
\min \left\Vert {X-{X_{0}}}\right\Vert
\]%
is equivalent the minimization problem
\[
\min \left\Vert {X-\frac{1}{2}\left( {{X_{0}}+X_{0}^{T}}\right) }\right\Vert
\]%
over the subset $S_{G}$ or $S_{E}$ in $S{\mathcal{R}^{n\times n}}$ since
\[
{\left\Vert {X-{X_{0}}}\right\Vert ^{2}}={\left\Vert {X-\frac{1}{2}\left( {{%
X_{0}}+X_{0}^{T}}\right) }\right\Vert }^{2}+{\left\Vert {\frac{1}{2}\left( {{%
X_{0}}-X_{0}^{T}}\right) }\right\Vert }^{2},\forall {X}\in S{\mathcal{R}%
^{n\times n}}.
\]%
(for example, see \cite{Liao, YXPeng, ZYPeng2} for details).
\end{rmk}

Now suppose that the skew--symmetric solutions of Problems \ref{prob1} and %
\ref{prob2} are required. To do this, the pair of matrix equations
\[
\begin{array}{l}
AXB=C \\
{B^{T}}X{A^{T}}=-{C^{T}}%
\end{array}%
\]%
or, equivalently, the usual system of linear equations
\[
\left[ {%
\begin{array}{c}
{{B^{T}}\otimes A} \\
{A\otimes {B^{T}}}%
\end{array}%
}\right] x=\left[ {%
\begin{array}{c}
{{c_{1}}} \\
-{{c_{2}}}%
\end{array}%
}\right]
\]%
instead of the matrix equation $AXB=C$ is taken where ${c_{1}}$ and ${c_{2}}$
are as in Remark \ref{rem:one}. Then, for a given matrix ${X_{0}}\in SS{%
\mathcal{R}^{n\times n}}$, the solution matrix $\hat{X}\in SS{\mathcal{R}%
^{n\times n}}$ is obtained by using $\hat{x} =vec\left( {\hat{X}}\right)$
given by
\begin{equation}
\hat{x}={\left[ {%
\begin{array}{c}
{{B^{T}}\otimes A} \\
{A\otimes {B^{T}}}%
\end{array}%
}\right] ^{\dag }}\left[ {%
\begin{array}{c}
{{c_{1}}} \\
-{{c_{2}}}%
\end{array}%
}\right] +{x_{0}}-{\left[ {%
\begin{array}{c}
{{B^{T}}\otimes A} \\
{A\otimes {B^{T}}}%
\end{array}%
}\right] ^{\dag }}\left[ {%
\begin{array}{c}
{{B^{T}}\otimes A} \\
{A\otimes {B^{T}}}%
\end{array}%
}\right] {x_{0}}.  \label{eq:coz2}
\end{equation}

\begin{rmk}
\label{rem:two} If the matrix $X_{0}$ is not skew--symmetric, then the
matrix ${\frac{1}{2}\left( {{X_{0}}-X_{0}^{T}}\right) }$\ instead of the
matrix $X_{0}$ is taken to find the skew--symmetric solutions of the
Problems \ref{prob1} or \ref{prob2}. The reason for this is that the
minimization problem%
\[
\min \left\Vert {X-{X_{0}}}\right\Vert
\]%
is equivalent the minimization problem
\[
\min \left\Vert {X-\frac{1}{2}\left( {{X_{0}}-X_{0}^{T}}\right) }\right\Vert
\]%
over the subset $S_{G}$ or $S_{E}$ in $SS{\mathcal{R}^{n\times n}}$ since
\[
{\left\Vert {X-{X_{0}}}\right\Vert ^{2}}={\left\Vert {X-\frac{1}{2}\left( {{%
X_{0}}-X_{0}^{T}}\right) }\right\Vert }^{2}+{\left\Vert {\frac{1}{2}\left( {{%
X_{0}}+X_{0}^{T}}\right) }\right\Vert }^{2},\forall {X}\in SS{\mathcal{R}%
^{n\times n}},
\]%
(for example, see \cite{Huang4} for details).
\end{rmk}

Note that the structures of the general solution and a least squares
solution of the matrix equation $AXB=C$ when the equation is consistent and
inconsistent, respectively, are exactly the same. Therefore, the structures
of the solutions of Problems \ref{prob1} and \ref{prob2} are the same, too.
The former and the latter facts are immediately seen from Theorem \ref%
{thm:one} and \ref{thm:two}, respectively, together with Lemma \ref%
{lem:three} in the framework of Lemmas \ref{lem:one} and \ref{lem:two}.

We conclude the paper by giving a few numerical examples. As it was
mentioned earlier, the matrices in the examples have been taken from
related reference work articles. It is seen that the solutions are %
exactly the same with four decimal digits as those in the works cited. %
All the computations have been performed using Matlab 7.5.

\begin{example}
\cite[Example 4]{Huang4}. Consider the skew--symmetric solution of Problem \ref
{prob1}, where
\[
A=\left[ {%
\begin{array}{ccccc}
1 & 3 & {-5} & 7 & {-9} \\
2 & 0 & 4 & 6 & {-1} \\
0 & {-2} & 9 & 6 & {-8} \\
3 & 6 & 2 & {27} & {-13} \\
{-5} & 5 & {-22} & {-1} & {-11} \\
8 & 4 & {-6} & {-9} & {-19}%
\end{array}%
}\right], \ \ B=\left[ {%
\begin{array}{ccccc}
4 & 0 & 8 & {-5} & 4 \\
{-1} & 5 & 0 & {-2} & 3 \\
4 & {-1} & 0 & 2 & 5 \\
0 & 3 & 9 & 2 & {-6} \\
{-2} & 7 & {-8} & 1 & {11}%
\end{array}%
}\right] ,
\]

\[
C=\left[ {%
\begin{array}{ccccc}
{171} & {-537} & {74} & {-29} & {-281} \\
{142} & {-278} & {212} & {-92} & {-150} \\
{196} & {-523} & {-59} & {-111} & {24} \\
{661} & {-1507} & {922} & {-234} & {-1003} \\
{-39} & {-192} & {-207} & {186} & {-227} \\
{-165} & {-292} & {-1154} & {76} & {422}%
\end{array}%
}\right], \ \ {X_{0}}=\left[ {%
\begin{array}{ccccc}
1 & 0 & 4 & {-1} & 0 \\
5 & 3 & 2 & 7 & 4 \\
{-1} & {-2} & 0 & {-1} & 0 \\
2 & 6 & 1 & 8 & {-4} \\
0 & 3 & 1 & 4 & 2%
\end{array}%
}\right] .
\]
By the formula (\ref{eq:coz2}) in the framework Remark \ref{rem:two}, the
skew--symmetric solution matrix is obtained as

\[
\hat{X}=\left[ {%
\begin{array}{ccccc}
{0.0000} & {2.0000} & {-1.0000} & {-2.0000} & {0.0000} \\
{-2.0000} & {0.0000} & {2.0000} & {1.0000} & {-4.0000} \\
{1.0000} & {-2.0000} & {0.0000} & {-1.0000} & {-0.0000} \\
{2.0000} & {-1.0000} & {1.0000} & {0.0000} & {-4.0000} \\
{0.0000} & {4.0000} & {-0.0000} & {4.0000} & {0.0000}%
\end{array}%
}\right] .
\]
\end{example}

\begin{example}
\cite[Example 1]{Liao}. Consider the symmetric solution of Problem \ref%
{prob2} where
\[
A=\left( {%
\begin{array}{cc}
{E_{55} } & {Z_{54}} \\
{Z_{45}} & {P_{4}}%
\end{array}%
}\right), \ \ B=\left( {%
\begin{array}{cc}
{K_{4}} & {Z_{45}} \\
{Z_{54}} & {Z_{55}}%
\end{array}%
}\right),
\]

\[
C=\left( {%
\begin{array}{cc}
{T_{4}} & {Z_{45}} \\
{Z_{54}} & {H_{5}}%
\end{array}%
}\right) ,\ \ {X_{0}}=\left( {%
\begin{array}{cc}
{I_{4}} & {{\textstyle{\frac{\mathrm{1}}{\mathrm{2}}}}{E_{45}} } \\
{{\textstyle{\frac{\mathrm{1}}{\mathrm{2}}}}{E_{54}} } & {I_{5}}%
\end{array}%
}\right) .
\]

Here $E_{mn}$ and $Z_{mn}$\ are $m\times n$ matrices whose all entries $1$
and $0,$ respectively, and $P_{4}$ and $H_{5}$ denote the $4\times 4$ symmetric
Pascal matrix and the $5\times 5$ Hilbert matrix, respectively. Moreover, $%
T_{4}$ and $K_{4}$ are Toeplitz and Hankel matrices given by

\[ T_{4}=\left[
\begin{array}{cccc}
1 & 2 & 3 & 4 \\
2 & 1 & 2 & 3 \\
3 & 2 & 1 & 2 \\
4 & 3 & 2 & 1%
\end{array}%
\right] \textit{ and }
K_{4}=\left[
\begin{array}{cccc}
1 & 2 & 3 & 4 \\
2 & 3 & 4 & 0 \\
3 & 4 & 0 & 0 \\
4 & 0 & 0 & 0%
\end{array}%
\right] ,
\]
respectively. By the formula (\ref{eq:coz1}) the symmetric
solution matrix is obtained as

\[
\hat{X}=\left[ {%
\begin{array}{ccccccccc}
{0.8258} & {-0.2692} & {-0.2480} & {-0.2214} & {0.4129} & 0 & 0 & 0 & 0 \\
{-0.2692} & {0.6358} & {-0.3430} & {-0.3164} & {0.3179} & 0 & 0 & 0 & 0 \\
{-0.2480} & {-0.3430} & {0.6783} & {-0.2952} & {0.3391} & 0 & 0 & 0 & 0 \\
{-0.2214} & {-0.3164} & {-0.2952} & {0.7314} & {0.3657} & 0 & 0 & 0 & 0 \\
{0.4129} & {0.3179} & {0.3391} & {0.3657} & 1 & 0 & 0 & 0 & 0 \\
0 & 0 & 0 & 0 & 0 & 1 & 0 & 0 & 0 \\
0 & 0 & 0 & 0 & 0 & 0 & 1 & 0 & 0 \\
0 & 0 & 0 & 0 & 0 & 0 & 0 & 1 & 0 \\
0 & 0 & 0 & 0 & 0 & 0 & 0 & 0 & 1%
\end{array}%
}\right] .
\]
\end{example}

\begin{example}
\cite[Example 3.1]{ZYPeng2}. Consider the symmetric solution of Problem \ref%
{prob2}, where
\[
A=\left[ {%
\begin{array}{ccccccc}
4 & 3 & {-1} & 3 & 1 & {-3} & 2 \\
3 & {-2} & 3 & {-4} & 3 & 2 & 1 \\
4 & 3 & {-1} & 3 & 1 & {-3} & 2 \\
3 & {-1} & 3 & {-1} & 3 & 2 & 1 \\
4 & 3 & {-1} & 3 & 1 & {-3} & 2 \\
3 & {-1} & 3 & {-1} & 3 & 2 & 1%
\end{array}%
}\right], \ \ B=\left[ {%
\begin{array}{cccccc}
{-3} & 4 & {-3} & {-3} & 4 & 4 \\
5 & {-3} & 5 & 5 & {-3} & {-3} \\
{-6} & 2 & {-6} & {-6} & 2 & 2 \\
{-8} & 4 & {-8} & {-8} & 4 & 4 \\
4 & {-5} & 4 & 3 & {-2} & {-7} \\
{-3} & 2 & {-3} & {-3} & 2 & 2 \\
{-1} & {-2} & {-1} & {-1} & {-2} & {-2}%
\end{array}%
}\right] ,
\]

\[
C=\left[ {%
\begin{array}{cccccc}
{43} & {-54} & {73} & {-54} & {51} & {-54} \\
{-31} & {37} & {-61} & {37} & {-53} & {37} \\
{43} & {-54} & {73} & {-54} & {51} & {-54} \\
{-31} & {37} & {-61} & {37} & {-53} & {37} \\
{47} & {-54} & {73} & {-54} & {21} & {-54} \\
{-31} & {27} & {-61} & {27} & {-53} & {27}%
\end{array}%
}\right], \ \ {X_{0}}=\left[ {%
\begin{array}{ccccccc}
{-1} & 2 & {-3} & 2 & {-1} & 1 & 3 \\
2 & {-1} & 3 & {-3} & 2 & {-3} & 4 \\
{-3} & 3 & {-3} & 3 & {-2} & 1 & {-1} \\
2 & {-3} & 3 & 2 & 2 & 2 & 4 \\
{-3} & 2 & {-2} & 2 & {-1} & 3 & {-3} \\
4 & 3 & 1 & 1 & 2 & 1 & 1 \\
1 & {-2} & 1 & 3 & 4 & {-1} & 1%
\end{array}%
}\right] .
\]
By the formula (\ref{eq:coz1}) in the framework Remark \ref{rem:one}, the
symmetric solution matrix is obtained as

\[
\hat{X}=\left[ {%
\begin{array}{ccccccc}
{\mathrm{1}\mathrm{.7699}} & {\mathrm{1}\mathrm{.8581}} & {\mathrm{-3}%
\mathrm{.5455}} & {\mathrm{2}\mathrm{.8924}} & {\mathrm{0}\mathrm{.5920}} & {%
\mathrm{0}\mathrm{.8523}} & {\mathrm{2}\mathrm{.3693}} \\
{\mathrm{1}\mathrm{.8581}} & {\mathrm{-0}\mathrm{.6722}} & {\mathrm{1}%
\mathrm{.8908}} & {\mathrm{-1}\mathrm{.9156}} & {\mathrm{3}\mathrm{.1173}} &
{\mathrm{0}\mathrm{.5698}} & {\mathrm{-1}\mathrm{.0561}} \\
{\mathrm{-3}\mathrm{.5455}} & {\mathrm{1}\mathrm{.8908}} & {\mathrm{-0}%
\mathrm{.5812}} & {\mathrm{1}\mathrm{.3562}} & {\mathrm{-3}\mathrm{.9861}} &
{\mathrm{1}\mathrm{.7472}} & {\mathrm{2}\mathrm{.2490}} \\
{\mathrm{2}\mathrm{.8924}} & {\mathrm{-1}\mathrm{.9156}} & {\mathrm{1}%
\mathrm{.3562}} & {\mathrm{-3}\mathrm{.8543}} & {\mathrm{-0}\mathrm{.6224}}
& {\mathrm{0}\mathrm{.3241}} & {\mathrm{3}\mathrm{.6833}} \\
{\mathrm{0}\mathrm{.5920}} & {\mathrm{3}\mathrm{.1173}} & {\mathrm{-3}%
\mathrm{.9861}} & {\mathrm{-0}\mathrm{.6224}} & {\mathrm{-2}\mathrm{.3618}}
& {\mathrm{-2}\mathrm{.2044}} & {\mathrm{3}\mathrm{.2716}} \\
{\mathrm{0}\mathrm{.8523}} & {\mathrm{0}\mathrm{.5698}} & {\mathrm{1}\mathrm{%
.7472}} & {\mathrm{0}\mathrm{.3241}} & {\mathrm{-2}\mathrm{.2044}} & {%
\mathrm{-0}\mathrm{.0556}} & {\mathrm{2}\mathrm{.7992}} \\
{\mathrm{2}\mathrm{.3693}} & {\mathrm{-1}\mathrm{.0561}} & {\mathrm{2}%
\mathrm{.2490}} & {\mathrm{3}\mathrm{.6833}} & {\mathrm{3}\mathrm{.2716}} & {%
\mathrm{2}\mathrm{.7992}} & {\mathrm{0}\mathrm{.0308}}%
\end{array}%
}\right] .
\]
\end{example}


\bigskip



\begin{thebibliography}{99}
\bibitem{Bjerhammer} A. Bjerhammer, Rectangular reciprocal matrices with special
reference to geodetic calculations, Kung. Tekn. Hogsk. Handl. Stockholm 45
(1951) 1--86.

\bibitem{Chu} K.E. Chu, Symmetric solutions of linear matrix equations by
matrix decompositions, Linear Algebra Appl. 119 (1989) 35--50.

\bibitem{Dai} H. Dai, On the symmetric solutions of linear matrix equations,
Linear Algebra Appl. 131 (1990) 1--7.

\bibitem{Don} F.J. Henk Don, On the symmetric solution of a linear matrix
equation, Linear Algebra Appl. 93 (1988) 1--7.

\bibitem{Graybill} F.A. Graybill, Introduction to matrices with applications in
statistics. Wadsworth Publishing Company inc., California, 1969.

\bibitem{Higham} N.J. Higham, Computing a nearest symmetric positive
semidefinite matrix, Linear Algebra Appl. 103 (1988) 103--118.

\bibitem{Horn} R.A. Horn and C.R. Johnson, Matrix Analysis. Cambridge
University Press, Cambridge, UK, 1985.

\bibitem{Huang1} G.X. Huang and F. Yin, Constrained inverse eigenproblem and
associated approximation problem for anti-Hermitian R-symmetric matrices,
Appl. Math. Comput. 186 (2007) 426--434.

\bibitem{Huang2} G.X. Huang and F. Yin, Matrix inverse problem and its optimal
approximation problem for R-symmetric matrices, Appl. Math. Comput. 189
(2007) 482--489.

\bibitem{Huang3} G.X. Huang, F. Yin and K. Guo, The general solutions on the
minimum residual problem and the matrix nearness problem for symmetric
matrices or anti-symmetric matrices, Appl. Math. Comput. 194 (2007) 85--91.

\bibitem{Huang4} G.X. Huang, F. Yin and K. Guo, An iterative method for the
skew-symmetric solution and the optimal approximate solution of the matrix
equation $AXB=C$, J. Comput. Appl. Math. 212 (2008) 231--244.

\bibitem{Jeseph} K.T. Jeseph, Inverse eigenvalue problem in structural design,
AIAA J. 30 (1992) 2890--2896.

\bibitem{Jiang} Z. Jiang and Q. Lu, Optimal application of a matrix under
spectral restriction, Math. Numer. Sinica 1 (1988) 47--52.

\bibitem{Liao} A. Liao and Y. Lei, Optimal approximate solution of the matrix
equation $AXB=C$ over symmetric matrices, J. Comput. Math. 25 (2007)
543--552.

\bibitem{Magnus} J.R. Magnus, L-structured matrices and linear matrix equation,
Linear Multilinear Algebra 14 (1983) 67--88.

\bibitem{Morris} G.R. Morris and P.L. Odell, Common solutions for n matrix
equations with applications, J. Assoc. Comput. Mach. 15 (1968) 272--274.

\bibitem{YXPeng} Y.X. Peng, X.Y. Hu and L. Zhang, An iteration method for the
symmetric solutions and the optimal approximation solution of the matrix
equation $AXB = C$, Appl. Math. Comput. 160 (2005) 763--777.

\bibitem{ZYPeng1} Z.Y. Peng, X.Y. Hu and L. Zhang, The inverse problem of
bisymmetric matrices, Numer. Linear Algebra Appl. 1 (2004) 59--73.

\bibitem{ZYPeng2} Z.Y. Peng, An iterative method for the least squares
symmetric solution of the matrix equation $AXB = C$, Appl. Math. Comput. 170
(2005) 711--723.

\bibitem{Rao} C.R. Rao and S.K. Mitra, Generalized inverse of matrices and
its applications. John Wiley, New York, 1971.
\end{thebibliography}
\end{document}